\documentclass[11pt]{amsart}
\usepackage{amsthm}
\usepackage{array}
\usepackage{amsmath}
\usepackage{enumerate}
\usepackage{tikz}
\usetikzlibrary{calc}
\usepackage{amssymb}
\usepackage{latexsym}
\usepackage{amsfonts}
\usepackage{color}
\usepackage{mathrsfs}
\usepackage{epsfig,helvet}

\theoremstyle{plain} \numberwithin{equation}{section}
\newtheorem{thm}{Theorem}[section]
\newtheorem{theorem}[thm]{Theorem}
\newtheorem{lemma}[thm]{Lemma}
\newtheorem{corollary}[thm]{Corollary}

\newtheorem{definition}[thm]{Definition}
\newtheorem{proposition}[thm]{Proposition}

\begin{document}
	\setcounter{page}{1}

	\title[ Hasan and Padhan]{Finite-dimensional nilpotent Lie superalgebras of class two and skew-supersymmetric bilinear maps}
	
	\author{IBRAHEM YAKZAN HASAN}
	\address{Centre for Applied Mathematics and Computing, Institute of Technical Education and Research  \\
		Siksha `O' Anusandhan (A Deemed to be University)\\
		Bhubaneswar-751030 \\
		Odisha, India}
	\email{ibrahemhasan898@gmail.com}
	\author[Padhan]{Rudra Narayan Padhan}
	\address{Centre for Data Science, Institute of Technical Education and Research  \\
		Siksha `O' Anusandhan (A Deemed to be University)\\
		Bhubaneswar-751030 \\
		Odisha, India}
	\email{rudra.padhan6@gmail.com, rudranarayanpadhan@soa.ac.in}

	\subjclass[2010]{Primary 17B30; Secondary 17B05.}
	\keywords{Nilpotent Lie superalgebras, skew-supersymmetric bilinear maps, generalized Heisenberg Lie superalgebras }
	\maketitle

\begin{abstract}
In this article, we discuss the category $\mathcal{SN}_2$ where the objects are 
finite-dimensional nilpotent Lie superalgebras of class two and the category 
$\mathcal{SSKE}$ where the objects are skew-supersymmetric bilinear maps. We establish relation between $\mathcal{SN}_2$ and $\mathcal{SSKE}$. As a result, we discuss the capability of nilpotent Lie superalgebras of class two.

%Let $F$ be a field of characteristic different from two. Suppose that $\mathcal{SN}_2$ is the category of finite-dimensional nilpotent Lie superalgebras of class two over the field $F$ and that $\mathcal{SSKE}$ is the category of skew-supersymmetric bilinear maps  $F$-vector superspaces. In this article, we establish a relation between the category $\mathcal{SN}_2$ and the category $\mathcal{SSKE}$. Then we show that the problem of determining the capability of these Lie superalgebras reduces to determining the epicenter of the corresponding objects in $\mathcal{SSKE}$. As an application of this technique, we determine the rank of the objects in $\mathcal{SSKE}$ which corresponding to generalized Heisenberg Lie superalgebras. Also we describe the structure of Lie superalgebras corresponding to someskew-supersymmetric bilinear maps of rank one.
\end{abstract}

\section{Introduction and Preliminaries}
Lie superalgebras have applications in many areas of Mathematics and Theoretical Physics as they can be used to describe supersymmetry. Kac \cite{Kac1977} gives a comprehensive description of mathematical theory of Lie superalgebras, and establishes the classification of all finite dimensional simple Lie superalgebras over an algebraically closed field of characteristic zero. In the last few years the theory of Lie superalgebras has evolved remarkably, obtaining many results in representation theory and classification. Most of the results are extension of well known facts of Lie algebras. But the classification of all finite dimensional nilpotent Lie superalgebras is still an open problem like that of finite dimensional nilpotent Lie algebras. 
	
\smallskip

% Baer \cite{Baer1938} defined the notion of capable group. Beyl et. al.,  \cite{Beyl1979} introduced the epicenter $Z^{*}(G)$ of a group $G$, and they proved that a group $G$ is capable if and only if $Z^{*}(G)=1$. Further exterior square of a group was studied for the first time in \cite{Brown2010}, which has an interesting relation with the capability of a group. Exterior center $Z^{\wedge}(G)$ of a group $G$ is defined as $Z^{\wedge}(G)=\{g \in G|~g\wedge h=1,~ \forall~h \in G\}$. Ellis \cite{Ellis1995} proved that $Z^{\wedge}(G)=Z^{*}(G)$. Similarly the notion of epicenter $Z^{*}(L)$ of a Lie algebra $L$ is given by Alamian et. al., \cite{Alamian2008}. The non-abelian tensor product, and exterior product of Lie algebras are defined, and some of the properties are studied by Ellis \cite{Ellis1987, Ellis1991, Ellis1995}. Recently, Niroomand et. al., \cite{PMF2013} investigated the connection between epicenter and exterior center of a finite dimensional Lie algebra. Finally, they have classified all capable Heisenberg Lie algebras, and in continuation, as an application they have shown that there exists at least one capable Lie algebra of arbitrary co rank.

Let $\mathbb{Z}_{2}=\{\bar{0}, \bar{1}\}$ be a field. A $\mathbb{Z}_{2}$-graded vector space $V$ is simply a direct sum of vector spaces $V_{\bar{0}}$ and $V_{\bar{1}}$, i.e., $V = V_{\bar{0}} \oplus V_{\bar{1}}$. It is also referred as a superspace. We consider all vector superspaces and superalgebras are over field $\mathbb{F}$ (characteristic of $\mathbb{F} \neq 2,3$). Elements in $V_{\bar{0}}$ (resp. $V_{\bar{1}}$) are called even (resp. odd) elements. Non-zero elements of $V_{\bar{0}} \cup V_{\bar{1}}$ are called homogeneous elements. For a homogeneous element $v \in V_{\sigma}$, with $\sigma \in \mathbb{Z}_{2}$ we set $|v| = \sigma$ is the degree of $v$. A  subsuperspace (or, subspace)  $U$ of $V$ is a $\mathbb{Z}_2$-graded vector subspace where  $U= (V_{\bar{0}} \cap U) \oplus (V_{\bar{1}} \cap U)$. We adopt the convention that whenever the degree function appears in a formula, the corresponding elements are supposed to be homogeneous. 
	
\smallskip 
	
A Lie superalgebra (see \cite{Kac1977, Musson2012}) is a superspace $L = L_{\bar{0}} \oplus L_{\bar{1}}$ with a bilinear mapping $ [., .] : L \times L \rightarrow L$ satisfying the following identities:
	\begin{enumerate}
		\item $[L_{\alpha}, L_{\beta}] \subset L_{\alpha+\beta}$, for $\alpha, \beta \in \mathbb{Z}_{2}$ ($\mathbb{Z}_{2}$-grading),
		\item $[x, y] = -(-1)^{|x||y|} [y, x]$ (graded skew-symmetry),
		\item $(-1)^{|x||z|} [x,[y, z]] + (-1)^{ |y| |x|} [y, [z, x]] + (-1)^{|z| |y|}[z,[ x, y]] = 0$ (graded Jacobi identity),
	\end{enumerate}
for all $x, y, z \in L$. Clearly $L_{\bar{0}}$ is a Lie algebra, and $L_{\bar{1}}$ is a $L_{\bar{0}}$-module. If $L_{\bar{1}} = 0$, then $L$ is just Lie algebra, but in general a Lie superalgebra is not a Lie algebra.  A Lie superalgebra $L$, is called abelian if  $[x, y] = 0$ for all $x, y \in L$. Lie superalgebras without the even part, i.e., $L_{\bar{0}} = 0$, are  abelian. A subsuperalgebra (or subalgebra) of $L$ is a $\mathbb{Z}_{2}$-graded vector subspace which is closed under bracket operation. The graded subalgebra $[L, L]$,  of $L$  is known as the derived subalgebra of $L$. A $\mathbb{Z}_{2}$-graded subspace $H$ is a graded ideal of $L$ if $[H, H]\subseteq L$. The ideal 
	\[Z(L) = \{z\in L : [z, x] = 0\;\mbox{for all}\;x\in L\}\] 
is a graded ideal and it is called the {\it center} of $L$. A homomorphism between superspaces $f: L \rightarrow G $ of degree $|f|\in \mathbb{Z}_{2}$, is a linear map satisfying $f(L_{\alpha})\subseteq G_{\alpha+|f|}$ for $\alpha \in \mathbb{Z}_{2}$. In particular, if $|f| = \bar{0}$, then the homomorphism $f$ is called homogeneous linear map of even degree. A Lie superalgebra homomorphism $f: L \rightarrow G$ is a  homogeneous linear map of even degree such that $f([x,y]) = [f(x), f(y)]$ holds for all $x, y \in L$.  If $H$ is an ideal of $L$, the quotient Lie superalgebra $L/H$ inherits a canonical Lie superalgebra structure such that the natural projection map becomes a homomorphism. The notions of {\it epimorphisms, isomorphisms} and {\it automorphisms} have the obvious meaning.

Throughout this article for superdimension of Lie superalgebra $L$ we simply write $\dim L=(m\mid n)$, where $\dim L_{\bar{0}} = m$ and $\dim L_{\bar{1}} = n$. Also  $A(m \mid n)$ denotes an abelian Lie superalgebra where $\dim A=(m\mid n)$. A  Lie superalgebra $L$ is said to be Heisenberg Lie superalgebra if $Z(L)=L'$ and $\dim Z(L)=1$. According to the homogeneous generator of $Z(L)$, Heisenberg Lie superalgebras can be further split into even or odd Heisenberg Lie superalgebras \cite{MC2011}. By Heisenberg Lie superalgebra we mean special Heisenberg Lie superalgebra in this article. 

\smallskip

Nayak defined Schur multiplier for Lie superalgebras \cite{Nayak2018}. If $L$ is a Lie superalgebra generated by a $\mathbb{Z}_{2}$-graded set $X =X_{\bar{0}} \cup X_{\bar{1}}$ and $\phi : X \rightarrow L$ is a degree zero map, then there exists a free Lie superalgebra $F$ and  $\psi: F \rightarrow L$ extending $\phi$. Let $R = \mbox{ker} (\psi)$. The extension, 
\begin{equation}
0 \longrightarrow R \longrightarrow F \longrightarrow L \longrightarrow 0
\end{equation} 
is called a {\it free presentation} of $L$ and is denoted by $(F, \psi)$ \cite{Musson2012}. With this free presentation of $L$, the {\it multiplier} of $L$ denoted by  $\mathcal{M}(L)$, is defined as
$$ 
\mathcal{M}(L) = \frac{[F,F]\cap R}{[F, R]}.
$$

For more details on Schur multiplier see \cite{GHLS2018, 2Nilpotent2020, Alamian2008, Batten1993,	Batten1996, Tappe, p1,Ellis1987, Ellis1991, Ellis1995,GKL2015, Hardy1998, Hardy2005, org, Kar1987, Nayak2018, SN2018b, N2, Niroomand2011, Russo2011, PMF2013, p50, hesam}. Now we list some useful results from \cite{Nayak2018}, for further use.

\begin{theorem} \label{th3.4}\cite[See Theorem 4.2, 4.3]{Nayak2018}
Every Heisenberg Lie superalgebra with even center has dimension $(2m+1 \mid n)$ and it is isomorphic to $H(m , n)=H_{\overline{0}}\oplus H_{\overline{1}}$, where
		\[H_{\overline{0}}=<x_{1},\ldots,x_{2m},z \mid [x_{i},x_{m+i}]=z,\ i=1,\ldots,m>\]
		and 
		\[H_{\overline{1}}=<y_{1},\ldots,y_{n}\mid [y_{j}, y_{j}]=z,\  j=1,\ldots,n>.\]
\end{theorem}

The multiplier and cover for Heisenberg Lie superalgebra of odd center is known.
	
\begin{theorem}\label{th3.6}\cite[See Theorem 2.8]{SN2018b}
Every Heisenberg Lie superalgebra, with odd center has dimension $(m \mid m+1)$ and it is isomorphic to $H_{m}=H_{\overline{0}}\oplus H_{\overline{1}}$, where
		\[H_{m}=<x_{1},\ldots,x_{m} , y_{1},\ldots,y_{m},z \mid [x_{j},y_{j}]=z,   j=1,\ldots,m>.\]
\end{theorem}

\begin{definition}
A Lie superalgebra $L$ is said to be $capable$ if there exists a Lie superalgebra $H$ such that $L \cong H/Z(H)$. 
\end{definition}

The epicenter $Z^{*}(L)$ is the smallest ideal of the Lie superalgebra $L$ such that $L/Z^{*}(L)$ is capable. Now, if $Z^{*}(L)$ is trivial then $L$ is capable. Classification of capable Lie superalgebras whose derived subalgebra have dimension atmost one, have been studied in \cite{Padhandetec, GHLS2018}. For more details on capability of groups, Lie algebras and Lie superalgebras see \cite{p10,Ellis1996, Beyl1979, Baer1938, 2Nilpotent2020, Alamian2008, p52, NPP, PMF2013, p50, p53,Padhantrip,Padhantensor,PadhanISO,PadhanHASDET,PadhanCOHOM,PadhanHOM,Padhan2MULT,PadhanSPLIT}. In this paper, we discuss the capability of nilpotent Lie superalgebras of class two.

\begin{theorem}\label{thm555}\cite[Theorem 6.7]{Padhandetec} 
		$H_m$ is capable if and only if $m = 1$.
\end{theorem}

\begin{theorem}\label{th4.22}\cite[ Theorem 6.3]{Padhandetec}
		$A(m \mid n)$ is capable if and only if $m=0, n=1$ or $m+n \geq 2$.
\end{theorem}

\begin{theorem}\label{th4.33}\cite[ Theorem 6.4]{Padhandetec}  
		$H(m, n)$ is capable if and only if $m=1, n=0$.
\end{theorem}

\begin{theorem}\label{th4.4}\cite[Proposition 3.4]{SN2018b}\cite[Theorem 6.9]{Padhandetec} \label{th5a}
Let $L$ be a nilpotent Lie superalgebra of dimension $(k \mid l)$ with $\dim L'=(r \mid s)$, where $r+s=1$. If $r=1, s=0$ then $L \cong H(m,n)\oplus A(k-2m-1 \mid l-n)$ for $m+n\geq 1$. If $r=0, s=1$ then $L \cong H_{m} \oplus A(k-m \mid l-m-1)$. Moreover, $L$ is capable if and only if either $L \cong H(1 , 0)\oplus A(k-3 \mid l)$ or $L \cong H_{1}\oplus A(k-1 \mid l-2)$.
\end{theorem}

\begin{lemma}\label{p11} \cite[Proposition 4.3]{2Nilpotent2020}
Let $M$ be a graded ideal of a Lie superalgebra $L=L_{\bar{0}}\oplus L_{\bar{1}}$. Then the following sequences are exact:
\begin{enumerate}
\item ${\rm{ker}}(\theta_c)\rightarrow \mathcal{M} ^{(c)}(L)\rightarrow \mathcal{M}^{(c)}(\frac{L}{M})\rightarrow \frac{M\cap \gamma_{c+1}(L)}{\gamma_{c+1}(M,L)}\rightarrow 0$.
\item If $M$ is $n$-central ideal with $n\leq c$, $M\wedge^c \frac{L}{\gamma_{n+1}(L)}\rightarrow \mathcal{M}^{(c)}(L)\xrightarrow{\varphi} \mathcal{M}^{(c)}(\frac{L}{M})\rightarrow M\cap \gamma_{c+1}(L)\rightarrow 0$.
\end{enumerate}
\end{lemma}

Let $L$ be a Lie superalgebra with a free presentation $0\longrightarrow R \longrightarrow F \overset{\pi}{\longrightarrow} L \longrightarrow 0$, where $F$ is a free Lie superalgebra and $R$ is an ideal of $F$. Then the Schur multiplier of $L$ can be defined as  $\mathcal{M}(L)=F^2 \cap R/[F,R]$. Let $H$ be a central ideal of a Lie superalgebra $L$. Then there exists an ideal $S$ in $F$ such that $H\cong S/R$ and $L/H\cong F/S$, so $\mathcal{M}(L/H)=F^2 \cap S/[F,S]$.  Suppose  $\pi(\tilde{x})=x$, where  $x \in L$ and $\tilde{x} \in F$, then the following are homomorphisms:
	
	\[\xi_1:\mathcal{M}(L)\longrightarrow \mathcal{M}(L/H)\] 
	\[x+[R,F]\longmapsto x+[S,F]\]
	
	\begin{equation}\label{eq2.1}
\varphi: (L/L^2) \otimes H \longrightarrow  \mathcal{M}(L) 
	\end{equation}
	\[(x+L^2)\otimes z\longmapsto [\tilde{z},\tilde{x}]+[R,F]\] 
	
	\[\rho:\mathcal{M}(L/H)\longrightarrow L^2\cap H\cong (S\cap F^2)/(R\cap F^2)\]
	\[x+[S,F]\longmapsto \pi(x)\]\\

\begin{lemma}\cite[ Proposition 5.7, Theorem 5.8 ]{Padhandetec}\label{lemm2.77}
Let $H$ be a central ideal of a finite-dimensional Lie superalgebra $L$. Then the following results hold.
\begin{enumerate}
\item The sequence $0\longrightarrow \mbox{ker}\rho\longrightarrow H\otimes \frac{L}{L^2}\overset{\varphi}\longrightarrow \mathcal{M}(L) \overset{\xi} \longrightarrow \mathcal{M}(L/H)\longrightarrow H\cap L^{2} \longrightarrow 0$ is exact.
\item  $H \subseteq Z^{*}(H)$ if and only if $\mbox{im} \varphi=0$.
\end{enumerate}
\end{lemma}

Let $L$ be a Lie superalgebra with a free presentation $0\longrightarrow R \longrightarrow F \overset{\pi}{\longrightarrow} L \longrightarrow 0$. If $x \in L$, then $\tilde{x}$ denotes a fixed pre-image of $x$ under $\pi$. It is easy to see that $L^2\cong (F^2+R)/R$ ,  $L/L^2\cong F/(F^2+R)$, so $\mathcal{M}(L/L^2)=F^2/(F^3+[R,F])$ and we have the following  homomorphisms:
	
\begin{equation}\label{eq1.3}
\beta: (L/L^2)\otimes L^2\longrightarrow \mathcal{M}(L) 	
\end{equation}
\[(x+L^2)\otimes z\longmapsto [\tilde{x},\tilde{z}]+[R,F]\]\\
		
\[\xi:\mathcal{M}(L)\longrightarrow \mathcal{M}(L/L^2)\] 
\[x+[R,F]\longmapsto x+(F^2+[R,F])\] \\
		
\[\zeta:\mathcal{M}(L/L^2) \longrightarrow L^2\] 
\[x+[R,F]\longmapsto \pi(x)\]

\begin{theorem}\label{theo28}
Let $L$ be a finite-dimensional nilpotent Lie superalgebra of class two, then the sequence
\[0\longrightarrow \mbox{ker}\beta\longrightarrow L^2 \otimes (L/L^2)\overset{\beta}\longrightarrow \mathcal{M}(L) \overset{\xi} \longrightarrow \mathcal{M}(L/L^2) \overset{\zeta} \longrightarrow  L^{2} \longrightarrow 0\] 
is exact. Moreover, 
$\mbox{ker}\beta=\langle(-1)^{|z||y|}((z+L^2)\otimes[x,y])+(-1)^{|x||z|}((x+L^2)\otimes[y,z])+(-1)^{|y||x|}((y+L^2)\otimes [z,x]) ~\lvert ~ x,y,z ~ \in L\rangle$.
\end{theorem}

\begin{proof}   
By taking $c =1$ in Lemma \ref{p11}, we have the following exact sequence
	\[0\longrightarrow \mbox{ker}\beta\longrightarrow L^2 \otimes (L/L^2)\overset{\beta}\longrightarrow \mathcal{M}(L) \overset{\xi} \longrightarrow \mathcal{M}(L/L^2) \overset{\zeta} \longrightarrow  L^{2} \longrightarrow 0\]
From the graded Jacobi identity:\\
	$\beta((-1)^{|z||y|}((z+L^2)\otimes[x,y])+(-1)^{|x||z|}((x+L^2)\otimes[y,z])+(-1)^{|y||x|}((y+L^2)\otimes [z,x]))= (-1)^{|\tilde{z}||\tilde{y}|}(\tilde{z}\otimes[\tilde{x},\tilde{y}])+(-1)^{|\tilde{x}||\tilde{z}|}(\tilde{x}\otimes[\tilde{y},\tilde{z}])+(-1)^{|\tilde{y}||\tilde{x}|}(\tilde{y}\otimes [\tilde{z},\tilde{x}])=0$. Hence the proof follows.
\end{proof}

As $ L/L^2$ is abelian Lie superalgebra, $\mathcal{M}(L/L^2)=(L/L^2)\wedge (L/L^2)$. Therefore, we obtain the following exact sequence 
\[0\longrightarrow \mbox{ker}\beta\longrightarrow L^2 \otimes (L/L^2)\overset{\beta}\longrightarrow \mathcal{M}(L) \overset{\xi} \longrightarrow L/L^2\wedge L/L^2 \overset{\zeta} \longrightarrow  L^{2} \longrightarrow 0\] \\
Using the above exact sequence, we have 
\[\mathcal{M}(L)\cong \mbox{ker}\xi\oplus \mbox{im}\xi=\mbox{im}\beta\oplus \mbox{ker}\zeta, \] \\
thus
\begin{equation}\label{eq21}
\mathcal{M}(L)\cong \frac{(L/L^2)\otimes L^2}{\mbox{ker}\beta}\oplus \mbox{ker}\zeta.
\end{equation}

Now we establish a relation between the kernel of $\beta$ as in \ref{eq1.3} and the epicenter of a Lie superalgebra.

\begin{lemma}\label{lemm2.9}
Let $H$ be a central ideal of a finite-dimensional non-abelian nilpotent Lie superalgebra $L$ of class two and let $\beta$ be as in \ref{eq1.3}. Then $H \subseteq Z^*(L)$ if and only if $H \subseteq Z(L)\cap L^2$ and $\langle(-1)^{|r||y|}(r+L^2)\otimes z \in (L/L^2)\otimes L^2) ~\lvert~ z=[x,y] \in H;~ x,y,r \in L\rangle \subseteq \mbox{ker}\beta.$
\end{lemma}

\begin{proof}
Let $H\subseteq Z^*(L)$. As $L$ is non-abelian nilpotent Lie superalgebra of class $2$, we can observe that $\dim L^2\leq \dim Z(L)\leq \dim L -2$, so by Theorem \ref{th4.22}, $L/L^2$ is capable Lie superalgebra. Now the definition of $Z^*(L)$ implies that $H\subseteq Z^*(L)\subseteq L^2$,  and in particular, $H \subseteq L^2$. If we consider the natural homomorphism $i :(L/L^2) \otimes H\longrightarrow(L/L^2) \otimes L^2$ and the map  $\varphi$ as in \ref{eq2.1}, then  $\mbox{im}(i) \subseteq (L/L^2) \otimes L^2$. Clearly, $\varphi =\beta i$, and hence $\beta((-1)^{|r||y|}(r+L^2) \otimes z) =\varphi((-1)^{|r||y|}(r+L^2) \otimes z)$ for every element $z=[x,y]\in H$. From Lemma \ref{lemm2.77}(ii), it follows that $\mbox{im}\varphi =0$, and so
	\[\beta i((L/L^2) \otimes H) = \varphi((L/L^2) \otimes H) = {0}.\]
Therefore, $\mbox{im}(i) \subseteq \mbox{ker} \beta$. Conversely, let $H\subseteq Z(L)\cap L^2$ and let $\mbox{im}(i) \subseteq \mbox{ker}\beta$. Since $\beta i((L/L^2) \otimes H) = \varphi((L/L^2) \otimes H)$, we have $\mbox{im}\varphi =0$. Thus Lemma \ref{lemm2.77}(ii) implies that $H\subseteq Z^*(L)$, as required.
\end{proof}

From the above lemma we can give another description of the epicenter of a nilpotent Lie superalgebra of class two.

\begin{corollary}\label{cor2.10}
Let $L$ be a nilpotent Lie superalgebra of class two, then
	\[Z^*(L)=\{z=[x,y] \in L^2 ~|~ (-1)^{|r||y|}(r+L^2)\otimes z \in \mbox{ker} \beta ~ {\rm{for ~ all~}} r \in L\}.  \]

\end{corollary}
%\begin{theorem}
%	Let $L,K$ be two nonabelian Lie superalgebras. Then 
%	\[Z^*(L\oplus K)= Z^*(L)\oplus Z^*(K).\]
%\end{theorem}
%\begin{corollary}
%	Let $L,K$ be two nonabelian Lie superalgebras. Then $L\oplus K$ is capable if and only if $L$ and $K$ are capable
%\end{corollary}

\begin{definition}
A finite dimensional Lie superalgebra $L$ is said to be generalized Heisenberg Lie superalgebra of rank $(m\mid n)$ if $Z(L) = L^2$ and $\dim L^2 = m + n$.
\end{definition}

\begin{theorem}\label{the2.14}\cite[Proposition 2.13]{GHLS2018} Let $L$ be an $(m\mid n)$dimensional nilpotent Lie superalgebra of nilpotency class two with $\dim L^2=(r\mid s )$.
Then $L = H \oplus A(l\mid k)$ and $Z^*(L) = Z^*(H)$, where $A(l\mid k)$ is  abelian lie superalgebra of dimension $(l \mid k)$ and $H$ is an $(m-l\mid n-k )$ dimensional generalized Heisenberg Lie superalgebra
superalgebra of rank $(r\mid s)$.
\end{theorem}

\section{Relation between the categories of Lie superalgebras of class two and skew-supersymmetric bilinear maps}

In this section, we introduce terminology and elementary properties for skew-supersymmetric bilinear maps, which will be used throughout this paper. Let $V=V_{\bar{0}} \oplus V_{\bar{1}}$ be a vector superspace, and  $Bil(V )$ denotes the space of all bilinear forms on $V$. It has a natural decomposition $Bil(V ) = (Bil V )_{\bar{0}} \oplus (Bil V )_{\bar{1}}$, where
\[ (Bil V )_{\bar{0}}=\{   f \in Bil (V )~|~ f(u, v) = 0 ~if~ |u| \neq|v|     \},       \]
\[ (Bil V )_{\bar{1}}=\{   f \in Bil( V) ~|~ f(u, v) = 0 ~if~ |u| \neq |v|+ 1     \}       .\]
Let $V$ and $W$ be two superspaces. A bilinear map $f:V\times V\longrightarrow W$ is said to be grading preserving map if $f(V_\alpha,V_\beta)\subseteq W_{\alpha+\beta}$ where $\alpha,\beta\in  \mathbb{Z}_{2}$. A bilinear map $f \in Bil(V ) $ is said to be supersymmetric (resp. skew-supersymmetric) if
\[f(u,v)=(-1)^{|u||v|}f(v,u)   \,\,\, (resp. \, f(u,v)=-(-1)^{|u||v|}f(v,u))\] 
for all homogeneous elements $u, v \in V $.

\begin{definition}
A bilinear map $f \in Bil(V ) $ is called non-degenerate if  $f(v, u) = 0$
for all $v \in  V$, then $u = 0$.
\end{definition}

Let $V$ and $W$ are two vector superspaces. Consider a skew-supersymmetric bilinear map $f:V\times V\longrightarrow W$. Put
$$f(V, V) = 	\{f(u, v) ~\lvert~ u,v \in V \}.$$
From here onwards, for every skew-supersymmetric bilinear map $f$, we assume that $W=f(V, V)$.

\smallskip
Let $\mathcal{SN}_2$ be the category of finite-dimensional nilpotent Lie superalgebras of class two over the field $\mathbb{F}$, and $\mathcal{SSKE}$ be the category whose objects are grading preserving skew-supersymmetric bilinear maps $f:V \times V\longrightarrow W$ such that $V$ and $W$ are non-trivial finite-dimensional superspaces over $\mathbb{F}$ and the image of $f$ spans $W$. The dimension of $W$ is the rank of $f$ and the dimension of $V$ is the dimension of $f$. In this section, we obtain an equivalence between the categories $\mathcal{SN}_2$ and $\mathcal{SSKE}$. Then we show that every object of the category $\mathcal{SN}_2$ corresponds to an object of the category $\mathcal{SSKE}$.

\begin{definition}\label{def3222}
An object $f:V\times V\longrightarrow W$ of the category  $\mathcal{SSKE}$ is called decomposable if $V$ can be decomposed into a non-trivial direct sum of two   
sub-superspaces $V=\overline{V_1}\oplus \overline{V_2}$ such that $f(v_1, v_2) =0$ for all $v_1 \in \overline{V_1}$ and $v_2 \in \overline{V_2}$. The subspace of $W$ spanned by $f(\overline{V_i}, \overline{V_i})$ is denoted by $\overline{W_i}$ for each $i =1, 2$. The map $\overline{f_i}:\overline{V_i}\times \overline{V_i}\longrightarrow \overline{W_i}$ is the restriction of $f$ to $\overline{V_i}\times \overline{V_i}$ for each $i =1, 2.$ In particular, $\overline{f_i}\in \mathcal{SSKE}$.
\end{definition}

Let $L$ be an object of $\mathcal{SN}_2$. Consider $U_L=L/L^2$ and $W_L=L^2$ as $\mathbb{F}$-superspaces. Define
\begin{equation}\label{eq11}
f_L: U_L\times U_L \longrightarrow W_L
\end{equation}\[(x+L^2,y+L^2)\longmapsto [x,y].\]
Then $f_L   \in \mathcal{SSKE}$. Moreover, the image of $f_L$ spans $L^2$.
Suppose $L_1$ and $L_2$ are two objects of $\mathcal{SN}_2$ and  $\delta:L_1\longrightarrow L_2$ is a Lie-homomorphism. Put
$U_{L_i}=L_i/L^2_i$ and $W_{L_i}=L^2_i$. We consider $U_{L_i}$ and $W_{L_i}$ as $\mathbb{F}$-superspaces for each $i =1, 2$. Now, we can define two Lie-homomorphisms
\begin{equation}\label{eq12}
\delta_1: U_{L_1}\longrightarrow U_{L_2}
\end{equation}
\[a+L^2_1\longmapsto \delta(a)+L^2_2,\]
and \[\delta_2=\delta\lvert_{W_{L_1}}:W_{L_1}\longrightarrow W_{L_2},\]
where $\delta\lvert_{W_{L_1}}$ denotes the restriction of $\delta$ on $L^2_1$ such that $\delta_2f_{L_1}=f_{L_2}(\delta_1\times\delta_1)$, i.e., the following diagram is commutative

\begin{center}
\begin{tikzpicture}[>=latex]
\node (x) at (0,0) {\(U_{L_1}\times U_{L_1}  \)};
\node (z) at (0,-3) {\(U_{L_2}\times U_{L_2}\)};
\node (y) at (3,0) {\(W_{L_1}\)};
\node (w) at (3,-3) {\(W_{L_2}\)};
\draw[->] (x) -- (y) node[midway,above] {$f_{L_1}$};
\draw[->] (x) -- (z) node[midway,left] {$\delta_1\times\delta_1$};
\draw[->] (z) -- (w) node[midway,below] {$f_{L_2}$};
\draw[->] (y) -- (w) node[midway,right] {$\delta_2$};
\end{tikzpicture}\\
\end{center}

Thus the pair $(\delta_1,\delta_2)$ is a morphism between $f_{L_1}$ and $f_{L_2}$.\\
\smallskip

For an object $f:V\times V\longrightarrow W$ in $\mathcal{SSKE}$, we can define a  Lie superalgebra $L_f$ in $\mathcal{SN}_2$ by choosing a basis $\{u_1, \ldots, u_m; v_1, \ldots,v_n\}$ for $V$ where $|u_i| =0$ for all $1\leq i\leq m$ ,  $|v_j|=1$ for all $1\leq j\leq n$, and a basis $\{w_1,\ldots, w_r; q_1,\ldots,q_s\}$ for $W$ where $|w_k| =0$ for all $1\leq k\leq r$ ,  $|q_t|=1$ for all $1\leq t\leq s$. Then $\{((u_i,0), 0),((0,v_j), 0),(0, (w_k,0)),(0, (0,q_t)) ~|~ 1 \leq i \leq m, 1 \leq j\leq n,1 \leq k\leq r,1 \leq t\leq s\}$ is a basis for the superspace $L_f=V\oplus W $. We can define a bracket $[.\, , .] :L_f\times L_f\longrightarrow L_f$ given by $$[((x, y), (w, l)),((x',y'),(w',l'))] =(0,f((x,y),(x',y'))),$$ where $x, y,x',y' \in V; w, l,w',l' \in W$. We have
 
\begin{equation}\label{eq333}
L_f=V\oplus W\cong \left\langle u_i,v_j,w_k,q_t ~|~ B(u_i,u_p)=[u_i,u_p], B(v_j,v_o)=[v_j,v_o],B(u_i,v_j)=[u_i,v_j], w_k,q_t \in Z(L_B) \right\rangle
\end{equation}
Now, it is clear that $L_f$ is an object of the category $\mathcal{SN}_2$.

\begin{theorem}\label{th33}
Define the functor $\mathcal{B}$ from the category $\mathcal{SN}_2$ to the category $\mathcal{SSKE}$ by
\[L \longrightarrow f_L \,\,\,\,\, ~ {\rm{for ~ each~object~}} L \, {\rm{in~}} \mathcal{SN}_2\]
\[\delta\longmapsto (\delta_1,\delta_2) \,\,\,\,\,\,~ {\rm{for ~ each~ morphism~}} \, \delta: L_1\longrightarrow L_2 \,\, {\rm{in~}}  \mathcal{SN}_2 \]
where the pair $(\delta_1,\delta_2)$ is defined in \ref{eq12} and the map $f_L$ is defined in \ref{eq11}. Then the following properties hold:
\begin{enumerate}
\item The functor $\mathcal{B}$ is an equivalence of the categories $\mathcal{SN}_2$ and $\mathcal{SSKE}$.
\item A morphism $\delta:L\longrightarrow H$ is an isomorphism in $\mathcal{SN}_2$ if and only if $\mathcal{B}(\delta)$is an isomorphism in $\mathcal{SSKE}$.
\end{enumerate}
\end{theorem}

\begin{proof}
(1) Let $L_1, L_2 \in \mathcal{SN}_2$. Put $U_i=L_i/L^2_i$, $W_i=L^2_i$, and the map $f_{L_i}$ as in \ref{eq11}. Let $(\delta_U, \delta_W)$ be a morphism from $f_{L_1}$ to $f_{L_2}$ such that $\delta_U:U_1\longrightarrow U_2$ and $\delta_W:W_1\longrightarrow W_2$. Suppose that $L_1\cong U_1\oplus W_1$ and $L_2\cong U_2\oplus W_2$ are as in \ref{eq333}. We obtain a Lie-homomorphism  $\varphi:L_1\longrightarrow L_2$ such that $\mathcal{B}(\varphi) =(\delta_U, \delta_W)$. Define
	\[\varphi:U_1\oplus W_1 \longrightarrow U_2\oplus W_2 \]
	\[((x, y), (w, l))\longmapsto (\delta_U((x, y)),\delta_W((w, l))) \]
	
It is clear that $\varphi$ is an even linear map. Let $((x_i, y_i), (w_i, l_i)) \in  U_1\oplus W_1$ for $i =1, 2$. Then
\begin{align*}
\varphi(((x_1, y_1), (w_1, l_1)),((x_2, y_2), (w_2, l_2)))& =(0,f_{L_1}((x_1, y_1),(x_2, y_2)  )\\
		&=(0,\delta_W(f_{L_1}((x_1, y_1),(x_2, y_2)  )))\\
		&=(0,f_{L_2}(\delta_U((x_1, y_1),(x_2, y_2)  )))\\
		&=[(\delta_U((x_1, y_1),\delta_W ((w_1, l_1))),(\delta_U((x_2, y_2)), \delta_W((w_2, l_2)))]\\
		&=[\varphi((x_1, y_1), (w_1, l_1)),\varphi((x_2, y_2), (w_2, l_2))].\\
\end{align*}
Therefore, $\varphi$ is a Lie-homomorphism such that $\mathcal{B}(\varphi)=(\delta_U, \delta_W)$. Thus, the functor $\mathcal{B}$ is full.\\
	
Consider $U_i=L_i/L^2_i$ and $W_i=L^2_i$ as $\mathbb{F}$-superspaces. It is clear that $\mathcal{B}$ induces the map $\mathcal{B}_1:Hom_{\mathcal{SN}_2}(L_1, L_2) \longrightarrow Hom_{\mathcal{SSKE}}(f_{L_1}, f_{L_2})$ given by $\delta\longmapsto(\delta_1,\delta_2)$ for each morphism $\delta:L_1\longrightarrow L_2$, where the pair $(\delta_1,\delta_2)$ is defined as in \ref{eq12}. Now we check that  $\mathcal{B}_1$ is injective.\\
 
 Let $\theta, \xi \in Hom_{\mathcal{SN}_2}(L_1, L_2)$ with $\mathcal{B}_1(\theta) =\mathcal{B}_1(\xi)$. Then $\theta_1(u+L^2)=\xi_1(u+L^2)$ and $\theta_2(a)=\xi_2(a)$, where $u \in L_1/L^2_1, a\in L^2_1$. If $\theta=0$ or $\xi=0$, then $\theta=\xi=0$. Now, suppose that $\theta$ and $\xi$ are non-zero homomorphisms. Consider $L_1\cong U_1\oplus W_1$ and $L_2\cong U_2\oplus W_2$ as superspaces. Let $\{x_1+L^2_1, \ldots, x_{n+m}+L^2_1\}$ be a basis of $U_1$ where $|x_i| =0$ for all $1\leq i\leq m$ ,  $|x_i|=1$ for all $m< i\leq m+n$  and let $\{y_1+L^2_2, \ldots, y_{r+s}+L^2_2\}$ be a basis of $U_2$  where $|y_j| =0$ for all $1\leq j\leq r$ ,  $|y_j|=1$ for all $r< j\leq r+s$. As image $\theta$ and image $\xi$ are contained in $L^2_2$, we conclude that $\theta$ and $\xi$ are mapped 	$\left\langle x_1, \ldots, x_{n+m} \right\rangle \cong U_1$  to $\left\langle y_1, \ldots, y_{r+s} \right\rangle \cong U_2$. Since $x_i+L^2_1\in U_1$ for all $i$ with $1 \leq i \leq m+n$, we obtain that $\theta(x_i)$ and $\xi(x_i)$ are two elements of the superspace $\left\langle y_1, \ldots, y_{r+s} \right\rangle  \cong U_2$ for all $i$ with $1 \leq i \leq m+n$. Now, using \ref{eq12}, we have $\theta(x_i)-\xi(x_i)\in L^2_2$ with $1 \leq i \leq m+n$. Hence, $\theta(x_i) =\xi(x_i)$ for all $i$ with $1 \leq i \leq m+n$. As $\theta(a) =\xi(a)$ for $a \in L^2_1$, we get $\theta=\xi$ as a linear map. Since $\theta([x_i, x_k]) =[\theta(x_i), \theta(x_k )] =[\xi(x_i), \xi(x_k )] = \xi([x_i, x_k ])$ for all $i, k$ with $1 \leq i, k\leq m+n$, therefore $\theta=\xi$ as a Lie-superhomomorphism. Thus, $\mathcal{B}_1$ is injective and the functor is faithful.\\
	 
Let $f \in \mathcal{SSKE}$. From the above argument of this theorem, there is a Lie superalgebra $L_f$ such that $\mathcal{B}(L_f) =f$. Thus every object of $\mathcal{SSKE}$ is isomorphic to an object of the form $\mathcal{B}(H)$ for some Lie superalgebra $H$ in  $\mathcal{SN}_2$. Therefore, the functor $\mathcal{B}$ is essentially surjective. Thus the functor $\mathcal{B}$ is an equivalence of the categories as it is full, faithful and essentially surjective.\\
	
 (2) Assume that $\delta:L\longrightarrow H$ is a Lie-superhomomorphism, and let $\mathcal{B}(\delta) =(\delta_1, \delta_2)$ be a morphism as in \ref{eq12}. Let $\delta$ be an isomorphism. Then $\delta_1$ and $\delta_2$ are both bijective maps. Thus $(\delta_1, \delta_2)$ is an isomorphism between $f_L$ and $f_H$. Conversely, the pair $(\delta_1, \delta_2)$ is an isomorphism between $f_L$ and $f_H$. Since $\delta_1$ is an isomorphism, we get $\mbox{ker}\delta\subseteq L^2$ and $H=\delta(L) +H^2$. As $\delta_2$ is an isomorphism, then $\mbox{ker}\delta=0$, and so $H^2=\delta(L^2) \subseteq \delta(L)$. Thus $\delta$ is an isomorphism.
\end{proof}

The following result is immediately obtained from Theorem \ref{th33}.

\begin{corollary}
The categories $\mathcal{SN}_2$ and $\mathcal{SSKE}$ are equivalent.
\end{corollary}

For fixed $m,n$ and $r,s$ with $1 \leq r < m$, $1 \leq s < n$, let $\mathcal{SN}_2(m \mid n, r\mid s)$ be a subcategory of the category $\mathcal{SN}_2$ whose objects are $(m \mid n)$-dimensional Lie superalgebras with derived subalgebra of dimension $(r\mid s)$ in $\mathcal{SN}_2$. Suppose that $\mathcal{SSKE}(m-r \mid n-s,r \mid s)$ is a subcategory of the category $\mathcal{SSKE}$ whose objects are objects of rank $(r \mid s)$ and dimension $(m-r \mid n-s)$ in $\mathcal{SSKE}$.

\begin{theorem}
Let $\tilde{\mathcal{B}}$ be a restriction of the functor $\mathcal{B}$ defined in Theorem \ref{th33} to the subcategory $\mathcal{SN}_2(m \mid n, r\mid s)$. Then the following properties hold:
\begin{enumerate}
\item The functor $\tilde{\mathcal{B}}$ is an equivalence of the categories $\mathcal{SN}_2(m \mid n, r\mid s)$ and $\mathcal{SSKE}(m-r \mid n-s,r \mid s)$.
\item A morphism $\delta:L\longrightarrow H$ is an isomorphism in $\mathcal{SN}_2(m \mid n, r\mid s)$ if and only if $\tilde{\mathcal{B}}(\delta)$ is an isomorphism in $\mathcal{SSKE}(m-r \mid n-s,r \mid s)$.
\end{enumerate}
\end{theorem}
\begin{proof}
The proof is similar to the proof of Theorem \ref{th33}.
\end{proof}

\begin{corollary}
The categories $\mathcal{SN}_2(m \mid n, r\mid s)$ and $\mathcal{SSKE}(m-r \mid n-s,r \mid s)$ are equivalent.
\end{corollary}

\section{Applications of equivalence categories $\mathcal{SN}_2$ and $\mathcal{SSKE}$}

In the previous section we have seen that the categories $\mathcal{SN}_2$ and $\mathcal{SSKE}$ are equivalent. In this section, we discuss some application of it, i.e., we discuss the Schur multiplier and the epicenter for the object of $\mathcal{SSKE}$.\\

%In this section, the operator of the usual tensor product of vector spaces is denoted by $\otimes$. We introduce the epicenter and the Schur multiplier for an object of $\mathcal{SSKE}$. Then we show that these concepts for a Lie superalgebra $L$ and $f_L$ as in \ref{eq11} are equal as sets. In the final of this section, we describe the structure of some Lie superalgebras in $\mathcal{SN}_2$ that corresponds to alternating bilinear maps in $\mathcal{SSKE}$. Now, we continue with the following definition.

Let $V=V_0\oplus V_1$ be a superspace. A  bilinear map $f \in Bil(V ) $ is called even (resp. odd) if  $f(V_{\bar{0}}, V_{\bar{1}}) = 0$, (resp. $f(V_{\bar{0}}, V_{\bar{0}}) = f(V_{\bar{1}}, V_{\bar{1}})=0)$. We denote the operator of the usual tensor product of vector spaces by $\otimes$. Let $f:V\times V\longrightarrow W$ be an object of $\mathcal{SSKE}$. We define 
$$X_f  = \left\langle(-1)^{|a||c|}a \otimes f(b, c) +(-1)^{|b||a|}b \otimes f(c, a) +(-1)^{|c||b|}c \otimes f(a, b)~|~ a, b, c\in V \right\rangle.   $$ 
Then $X_f$ is supersubspace of $V\otimes W$. Let $\rho:V\wedge V\longrightarrow W$ be defined by $$v_1\wedge v_2\longmapsto B(v_1,v_2).$$ Then $\rho$ is a surjective (skew-symmetric)bilinear map. Now, we define the Schur multiplier of $f$ as follows:
\[\mathcal{M}(f)=\frac{U\otimes W}{X_f}\oplus \mbox{ker}\rho.\]

By using \ref{eq21} and the Theorem \ref{theo28}, we have the following proposition.

\begin{proposition}
Let $L$ be a Lie superalgebra in $\mathcal{SN}_2$ and let $f_L$ be defined in \ref{eq11}. Then
 \[\mathcal{M}(L)=\mathcal{M}(f_L).\]
\end{proposition}
 
From the above proposition we can see that the multiplier of $L$ which is an object of  $\mathcal{SN}_2$ is equal to the multiplier of $f_L$ which is an object of $\mathcal{SSKE}$. Now we state the epicenter for an object of $\mathcal{SSKE}$, then we will relate it to $\mathcal{M}(f_L)$.

\begin{definition}\label{def44}
Let $f:V\times V\longrightarrow W$ be an object of $\mathcal{SSKE}$. We define the epicenter of $f$ as follows:
\[Z^*(f)=\{w=f(u,v) \in W~ \lvert ~(-1)^{|x||v|}x\otimes w \in X_f  ~ {\rm{for ~ all~homoenuos~elements ~}} x,u,v \in V \}.\]
\end{definition}

\begin{proposition}\label{pro45}
Let $L \in \mathcal{SN}_2$ and let $f_L$ be defined in \ref{eq11}. Then
\[Z^*(L)=Z^*(f_L).\]
\end{proposition}

\begin{proof}
Corollary \ref{cor2.10} implies $Z^*(L) =\{z=[x,y] \in L^2|\, (-1)^{|r||y|}(r+L^2)\otimes z \in \mbox{ker}\beta \,\, {\rm{for ~ all~ }} r\in L\}$. It is clear from Theorem \ref{theo28} that $X_{f_L}=\mbox{ker}\beta$. Now the result follows from the definition of $Z^*(f)$.
\end{proof}

Now it is clear that to determine the capability of an object of $\mathcal{SN}_2$, it is enough to check that $Z^*(f_L)$ is trivial space.

\begin{corollary}
Let $L \in \mathcal{SN}_2$. Then $L$ is capable if and only if $Z^*(f_L)=0$.
\end{corollary}

From the above corollary we can observe that a nilpotent Lie superalgebra of class two is capable if and only if $Z^*(f_L)=0$.

\begin{lemma}\label{lem4.7}
Let $f:V\times V\longrightarrow W$ be an object of $\mathcal{SSKE}$ and let $I$ be a subspace of $W$. Then $\bar{f}:V\times V\longrightarrow W/I$ given by $\bar{f}(x, y) =f(x, y) + I$ is an object of $\mathcal{SSKE}$ and $(Z^*(f) +I)/I\subseteq Z^*(\bar{f})$.
\end{lemma}

\begin{proof}
Consider the map $\bar{f} :V \times V \longrightarrow W/I $ defined by
                \[\bar{f}(x, y) = f(x, y) + I,\]
and	$\eta :V \otimes W \longrightarrow V \otimes (W/I) $ given by 
         \[ x \otimes y \longmapsto x \otimes (y + I),\]
and the restriction of $\eta$ to $X_f$ $\eta\lvert_{X_f}: X_f\longrightarrow V\otimes (W/I)$ defined by
\[(-1)^{|a||c|}a \otimes f(b, c) +(-1)^{|b||a|}b \otimes f(c, a) +(-1)^{|c||b|}c \otimes f(a, b)\longmapsto\]
\[(-1)^{|a||c|}a \otimes (f(b, c)+I) +(-1)^{|b||a|}b \otimes (f(c, a)+I) +(-1)^{|c||b|}c \otimes (f(a, b)+I).\]
One can observe that $\bar{f}\in \mathcal{SSKE}$, $\eta$ is a homomorphism, and $\mbox{im}(\eta\lvert_{X_f}) =X_{\bar{f}}$. As $w=f(u,v)\in Z^*(f)$, then $(-1)^{|x||v|}x\otimes w\in X_f$ for all $x \in V$. Now $\eta\lvert_{X_f}((-1)^{|x||v|}x\otimes w) =(-1)^{|x||v|}x\otimes (w+I) \in X_{\bar{f}}$ for all $x \in V$. Therefore $w+I\in Z^*(\bar{f})$.
\end{proof}
\smallskip

%The following result is an easy consequence of Proposition \ref{pro45} and Lemma \ref{lem4.7}.

\begin{corollary}
Let $L \in \mathcal{SN}_2$ and let $H$ be a graded central ideal of $L$ such that $H \subseteq L^2$. Then
	 \[ (Z^*(L) +H)/H\subseteq Z^*(L/H) \].
\end{corollary}

\begin{proof}
The proof follows from Proposition \ref{pro45} and Lemma \ref{lem4.7}.
\end{proof}

%We will show that there is one class of objects of rank $(\binom m2+\binom {n+1}{2}\mid mn)$ and dimension $(m \mid n)$ of $\mathcal{SSKE}$, up to isomorphism.

\begin{theorem}\label{theo49}
Let $f:V\times V\longrightarrow W$ be an object of rank $(\binom m2+\binom {n+1}{2}\mid mn)$  and dimension $(m \mid n)$ of $\mathcal{SSKE}$. Then $L_f$ is an $(\binom {m+1}2+ \binom {n+1}{2}\mid mn+n)$-dimensional generalized Heisenberg Lie superalgebra of rank $(\binom m2+\binom {n+1}{2}\mid mn)$, where $L_f$ is the Lie superalgebra associated to $f$.
\end{theorem}

\begin{proof}
Let $\{u_1, \ldots, u_m;v_1, \ldots,v_n\}$ be a basis for $V=V_{\bar{0}}\oplus V_{\bar{1}}$ where $|u_i| =0$ for all $1\leq i\leq m$ ,  $|v_j|=1$ for all $1\leq j\leq n$. Let $x_{ir}:=f(u_i, u_r)\in W_{\bar{0}}$, $y_{js}:=f(v_j, v_s)\in W_{\bar{0}}$, and $z_{ij}:=f(u_i, v_j)\in W_{\bar{1}}$,  $1\leq j\leq s\leq n$, $1\leq i<r\leq m$. Then 
\[ W=< x_{ir},y_{js},z_{ij}~ \lvert ~ 1\leq j\leq s\leq n, ~ 1\leq i<r\leq m  >.\]
Since $\dim W=(\binom m2+\binom {n+1}{2}\mid mn)$, thus $\{ x_{ir},y_{js},z_{ij}~\lvert ~1\leq j\leq s\leq n,~ 1\leq i<r\leq m \}$ is a basis of $W$. Now
\[ L_f\cong   \left\langle  u_i,v_j,  x_{ir},y_{js},z_{ij} ~\lvert ~[u_i,u_r]=x_{ir},~ [v_j,v_s]=y_{js},~ [u_i,v_j]=z_{ij},~1\leq j\leq s\leq n,~ 1\leq i<r\leq m \right\rangle.  \]
Thus $L_f$ is an $(\binom {m+1}2+ \binom {n+1}{2}\mid mn+n)$-dimensional generalized Heisenberg Lie superalgebra of rank $(\binom m2+\binom {n+1}{2}\mid mn)$.
\end{proof}

Now we have the following results for the even and odd objects of $\mathcal{SSKE}$.\\

\begin{corollary}
Let $f:V\times V\longrightarrow W$ be an even object of rank $(\binom m2+\binom {n+1}{2}\mid 0)$  and dimension $(m \mid n)$ of $\mathcal{SSKE}$. Then $L_f$ is an $(\binom {m+1}2+ \binom {n+1}{2}\mid n)$-dimensional generalized Heisenberg Lie superalgebra of rank $(\binom m2+\binom {n+1}{2}\mid 0)$. 
\end{corollary}
\begin{proof}
Since $f$ is even, so $z_{ij}=f(u_i, v_j)=0$ in the Theorem \ref{theo49}.
\end{proof}

\begin{corollary}
Let $f:V\times V\longrightarrow W$ be an odd object of $\mathcal{SSKE}$ of rank $(0\mid mn)$ and dimension $(m \mid n)$. Then $L_f$ is an $(m\mid mn+n)$-dimensional generalized Heisenberg Lie superalgebra of rank $(0\mid  mn)$. 
\end{corollary}
\begin{proof}
Since $f$ is odd, $x_{ir}=f(u_i, u_r)=0,~ y_{js}=f(v_j, v_s)=0$, and the proof follows from Theorem  \ref{theo49}.
\end{proof}

%By Definition \ref{def3222}, for every decomposable object $f$ of $\mathcal{SSKE}$, the following result shows that the set $\mbox{im}\bar{f_1}\bigcup \mbox{im}\bar{f_2}$ spans $W$. We consider $\bar{V_i}\otimes\bar{W_j}$ as a subspace of $V\otimes W$ for all $1 \leq i, j\leq2$.

\begin{theorem}\label{the3.13}
Let $f:V\times V\longrightarrow W$ be a decomposable object of $\mathcal{SSKE}$ with the notations used in Definition \ref{def3222}. Then the following properties hold:
\begin{enumerate}
		\item $W=\overline{W_1}+\overline{W_2}$.
		\item $X_f=X_{\bar{f_1}}+X_{\bar{f_2}}+\overline{V_1}\otimes \overline{W_2}+\overline{V_2}\otimes \overline{W_1}.$
		\item
		 Assume that $\dim W=2$, that $W_1=W=\overline{W_2}\oplus \overline{U}$, and that $\tilde{f}_1:\overline{V_1}\times \overline{V_1}\longrightarrow \overline{U}$ is a decomposable object in $\mathcal{SSKE}$. If $f$ is nondegenerate, then 
		\[X_f = X_{\tilde{f}_1} + V \otimes \overline{W_2} \oplus \overline{V_2} \otimes \overline{U}.\] 
\end{enumerate}
\end{theorem}
\begin{proof}
(1) Let $V=\overline{V_1}\oplus \overline{V_2}$. Then $\overline{W_1}+\overline{W_2}\subseteq W $. Now, let $u, v\in V$, then $u =u_1+u_2$, $v=v_1+v_2$ where $u_1,~ v_1\in \overline{V_1} $ and $u_2,~ v_2\in \overline{V_2}$.Then $f(u, v) =f(u_1, v_1) +f(u_2, v_2) =\bar{f_1}(u_1, v_1) +\bar{f_2}(u_2, v_2)$, and thus $W \subseteq \overline{W_1}+\overline{W_2}$. Therefore, $W = \overline{W_1}+\overline{W_2}$.	\\
(2) Observe that $X_{\bar{f_1}}+X_{\bar{f_2}}\subseteq X_f$. At first, we prove
$X_{\bar{f_1}}+X_{\bar{f_2}} + \overline{V_1}\otimes\overline{W_2}+\overline{V_2}\otimes\overline{W_1}\subseteq X_f$, for which we only need to prove that $\overline{V_1}\otimes\overline{W_2}+\overline{V_2}\otimes\overline{W_1}\subseteq X_f$. Let $u_1,~ u_2\in \overline{V_1}$ and $x, ~y\in \overline{V_2} $. Since $f$ is decomposable, we have

\begin{align*}
(-1)^{|u_1||y|}u_1 \otimes f(x, y) +(-1)^{|x||u_1|}x \otimes f(y, u_1) +(-1)^{|y||x|}y\otimes f(u_1, x) &=(-1)^{|u_1||y|}u_1 \otimes f(x, y)\\
&=(-1)^{|u_1||y|}u_1 \otimes \bar{f_2}(x, y)
\end{align*}
and 
\begin{align*}
(-1)^{|x||u_2|} x \otimes f(u_1, u_2) +(-1)^{|u_1||x|}u_1 \otimes f(u_2, x) +(-1)^{|u_2||u_1|}u_2\otimes f(x, u_1) &=(-1)^{|x||u_2|}x \otimes f(u_1, u_2)\\
&=(-1)^{|x||u_2|}x \otimes \bar{f_1}(u_1, u_2).
\end{align*}
Thus $X_{\bar{f_1}}+X_{\bar{f_2}}+\overline{V_1}\otimes\overline{W_2}+\overline{V_2}\otimes\overline{W_1}\subseteq X_f.$ For the other containment, let $z=(-1)^{|a||c|}a \otimes f(b, c) +(-1)^{|b||a|}b \otimes f(c, a) +(-1)^{|c||b|}c \otimes f(a, b)\in  X_f$, where $a, b, c$ are homogeneous elements in $V$. Since  $f(\overline{V_1},\overline{V_2})=0$, we have the following cases:\\
	(i) If $a,~b,~c$ are homogeneous elements in $\overline{V_1}$, then $z \in X_{\bar{f_1}}$.\\
	(ii) If $a,~b,~c$ are homogeneous elements in $\overline{V_2}$ then $z \in X_{\bar{f_2}}.$\\
	(iii) If one of $a,b,c$ say $a$ is homogeneous element of $\overline{V_1}$ and $b,c$ are homogeneous elements of $\overline{V_2}$ or $a$ is homogeneous element of $\overline{V_2}$ and $b,c$ are homogeneous elements of $\overline{V_1}$. Then 
	\[z=(-1)^{|a||c|}a \otimes f(b, c)=(-1)^{|a||c|}a \otimes \bar{f_2}(b, c) \in \overline{V_1}\otimes\overline{W_2},\]
	or  
	\[z=(-1)^{|a||c|}a \otimes f(b, c)=(-1)^{|a||c|}a \otimes \bar{f_1}(b, c) \in \overline{V_2}\otimes\overline{W_1}.\]
From (i), (ii) and (iii), we have $z \in X_{\bar{f_1}}+X_{\bar{f_2}}+\overline{V_1}\otimes \overline{W_2}+\overline{V_2}\otimes \overline{W_1} $.
	Therefore 
$X_f\subseteq X_{\bar{f_1}}+X_{\bar{f_2}}+\overline{V_1}\otimes \overline{W_2}+\overline{V_2}\otimes \overline{W_1}$, and so 
\[X_f= X_{\bar{f_1}}+X_{\bar{f_2}}+\overline{V_1}\otimes \overline{W_2}+\overline{V_2}\otimes \overline{W_1}.\]

(3) Suppose that $W= \left\langle w_1 \right\rangle \oplus \left\langle w_2 \right\rangle$, $\overline{W_2}=\left\langle w_2 \right\rangle$, and that $\overline{U}=\left\langle w_1 \right\rangle$. Let $g_i$ and $h$ be graded alternating bilinear forms on $\overline{V_1}$ and $\overline{V_2}$ for each $i =1, 2$, respectively. Then $\bar{f}_1=g_1w_1+g_2w_2$ and $\bar{f}_2=h w_2$. Put $\tilde{f}_1=g_1w_1$. Now, we are about to show that $X_f = X_{\tilde{f}_1} + V \otimes \overline{W_2} \oplus \overline{V_2} \otimes \overline{U}$. Using part (ii) and the fact $X_{\bar{f}_2}\subseteq \overline{V_2}\otimes \overline{W_2}\subseteq\overline{V_2}\otimes \overline{W_1}$ , we obtain that
	\begin{align*}
	X_f= 
	& X_{\bar{f_1}}+X_{\bar{f_2}}+\overline{V_1}\otimes \overline{W_2}+\overline{V_2}\otimes \overline{W_1}\\
	&=X_{\bar{f_1}}+\overline{V_1}\otimes \overline{W_2}+\overline{V_2}\otimes \overline{W}\\
	&=X_{\bar{f_1}}+\overline{V_1}\otimes \overline{W_2}+\overline{V_2}\otimes (\overline{W_2}\oplus \overline{U})\\
	&=X_{\bar{f_1}}+V\otimes \overline{W_2}+\overline{V_2}\otimes\overline{U}.\\
\end{align*}
We claim that $X_{\bar{f}_1}\subseteq X_{\tilde{f}_1}
+V\otimes \overline{W_2}$. For $x, y, z \in \overline{V_1}$, since
 $\bar{f}_1=g_1w_1+g_2w_2$, we have
\begin{align*}
	&(-1)^{|x||z|}x \otimes \bar{f_1}(y, z)+(-1)^{|y||x|}y \otimes \bar{f_1}(z, x)+(-1)^{|z||y|}z \otimes \bar{f_1}(x, y) \\
	&=(-1)^{|x||z|}x \otimes (g_1(y, z)w_1+g_2(y, z)w_2)+(-1)^{|y||x|}y \otimes (g_1(z, x)w_1+g_2(z, x)w_2)\\
	&+(-1)^{|z||y|}z \otimes (g_1(x, y)w_1+g_2(x, y)w_2)\\
	&=(-1)^{|x||z|}x \otimes \tilde{f_1}(y, z)+(-1)^{|y||x|}y \otimes \tilde{f_1}(z, x)+(-1)^{|z||y|}z \otimes \tilde{f_1}(x, y)+(-1)^{|x||z|}x \otimes g_2(y, z)w_2\\
	&+(-1)^{|y||x|}y \otimes g_2(z, x)w_2+(-1)^{|z||y|}z \otimes g_2(x, y)w_2 \in X_{\tilde{f}_1}
	+\overline{V_1}\otimes \overline{W_2}\subseteq X_{\tilde{f}_1}
	+V\otimes \overline{W_2}.
\end{align*}
Therefore $X_{\bar{f}_1}\subseteq X_{\tilde{f}_1}
+V\otimes \overline{W_2}$, then we have 
\[ X_f =X_{\bar{f_1}}+V\otimes \overline{W_2}+\overline{V_2}\otimes\overline{U}\subseteq X_{\tilde{f}_1}+V\otimes \overline{W_2}+\overline{V_2}\otimes\overline{U}.\]
On the other hand, since $X_{\tilde{f}_1}\subseteq X_{\bar{f}},$ then 
\[X_f = X_{\tilde{f}_1} + V \otimes \overline{W_2} \oplus \overline{V_2} \otimes \overline{U}.\]
\end{proof}

In the following theorem, we classify the objects of rank one in $\mathcal{SSKE}$.

\begin{theorem}\label{the3.14}
Let $f:V\times V\longrightarrow W$ be a decomposable object of $\mathcal{SSKE}$ of rank $(r \mid s)$ where $r+s=1$. Then the following properties hold:
	\begin{enumerate}
		\item if $r=1, s=0$ then $L_f\cong H(m\mid n) \oplus A(k,l)$ for some $m,n,k$ and $l$.
		\item If $r=0, s=1$ then $L_f\cong H(m) \oplus A(k',l')$ for some $m,k'$ and $l'$.
		\item There are two subspaces $D_1$ and $D_2$ of $V$ such that $V=D_1\oplus D_2$. Suppose that $S_i=f(D_i, D_i)$ and $f_i=f|_{D_i\times D_i}:D_i\times D_i\longrightarrow S_i$ for each $i =1, 2$. Then $f=f_1+f_2$, $S_2=0$, and $f_1$ is a nondegenerate object of rank one in ALT. If $\dim D_1=(2|0)$ or $\dim D_1=(1|1)$ or $\dim D_1=(0|2)$ then $X_f=D_2\otimes W$. Otherwise, $X_f=U\otimes W$.
\end{enumerate}  
Where $L_f$ is the Lie superalgebra associated to $f$.
\end{theorem}
\begin{proof}
	(1) Let $L_f$ be the Lie superalgebra associated to $f$. Using Theorem \ref{the2.14}, we have $L_f\cong H\oplus A$, where $H$ is a generalized Heisenberg Lie superalgebra and $A$ is an abelian Lie superalgebra. Obviously, the rank of $H$ is equal to the rank of $f$, which is equal to $(1\mid 0)$. Hence, $L_f=H(m\mid n) \oplus A(k,l)$ for some $m,n,k$ and $l$.\\
	(2) Similarly we have  $L_f\cong H\oplus A$, where $H$ is a generalized Heisenberg Lie superalgebra and $A$ is an abelian Lie superalgebra. And the rank of $H$ is equal to the rank of $f$, which is equal to $(0\mid 1)$. Hence, $L_f=H(m) \oplus A(k',l')$ for some $m,k'$ and $l'$.\\
	(3) By part (1), we get $L_f=H(m,n) \oplus A(k|l)$ for some $m,n$ and $k,l$, and by Theorem \ref{th4.4}, also $L_f$ is capable if and only if $m =1, n=0$. Put $V=L_f/L^2_f$, $D_1=H(m,n)/H(m,n)^2$, and $D_2=A(k|l)$. Since $L_f/L^2_f=(H(m,n)/H(m,n)^2) \oplus A(k|l)$, we have $V=D_1\oplus D_2$. Let now $S_i=f(D_i, D_i)$ and $f_i=f|_{D_i\times D_i}:D_i\times D_i\longrightarrow S_i$ for each $i =1, 2$. It is clear that $f_1\in \mathcal{SSKE}$ and $S_2=0$. As $L_{f_1}$ is a generalized Heisenberg Lie superalgebra of rank one, so $f_1$ is a nondegenerate object of rank one in $\mathcal{SSKE}$. Similar to the proof of Theorem \ref{the3.13}, we have
\begin{eqnarray}\label{eq3.11}
	X_f=X_{f_1}+X_{f_2}+D_1\otimes S_2+D_2\otimes S_1.
\end{eqnarray}
If $\dim D_1=(r|s)$ with $r+s=2,$ we have
$$
  L_f\cong
	\begin{cases}
		H(1|0)	\oplus	A(k|l) \quad \mbox{if}\;r=2,~s=0\\
		H(0|2) \oplus	A(k'|l') \quad \mbox{if}\;r=0,~ s=2\\  
		H(1)\oplus A(k''|l'') \quad \mbox{if}\; r=1, ~s=1.\\  
	\end{cases}
$$

Since $S_2=0$, we get $X_{f_2}=0$ and $D_1\otimes S_2=0$. Clearly, $X_{f_1}=0$. Using \ref{eq3.11} and the fact $S_1=W$, we get $X_f=D_2\otimes W$. If $\dim D_1\neq 2$, then $L_f$ is noncapable. Using Proposition \ref{pro45} and Theorem \ref{the2.14}, since $Z^*(f_1) =Z^*(H(m,n))=Z^*(L_f)=W$, we have $X_{f_1}=D_1\otimes W$, and so $X_f=D_1\otimes W+D_2\otimes W=V\otimes W$. Similarly if $L_f\cong H(m) \oplus A(k',l')$.
\end{proof}

\begin{definition}
The Lie superalgebra $L$ is called a central sum of two supersubalgebras $K$ and $H$, if $L =K+H$, where $[H,K]=0$ and $H\cap K \subseteq Z(L)$. We denote it by $L = H \diamond K$.
\end{definition}

The proof of the following results are analogous to that of Lie algebras (see \cite{Johari21}), so we omit the proof.

\begin{lemma}\label{lem3.161}
Let $f:V\times V\longrightarrow W$ be a decomposable nondegenerate object in $\mathcal{SSKE}$ with the notations used in Definition \ref{def3222}. Then the following results hold:
	\begin{enumerate}
		\item $L_f=L_{\overline f_1} \diamond L_{\overline f_2}$ and   
		$L_{\overline f_1} \cap L_{\overline f_2}=L^2_{\overline f_1} \cap L^2_{\overline f_2}\subseteq Z(L_f)$, where $L_f$ and $L_{\overline f_i}$ are generalized Heisenberg Lie superalgebras associated to $f$ and $\overline f_i$ for each $i =1, 2$, respectively.
		
		\item The sum in part (1) is direct if and only if $\overline W_1\cap \overline W_2=0$. Also, $Z^*(f) =Z^*(\overline f_1) \oplus Z^*(\overline f_2)$.
		
	\end{enumerate}
	
\end{lemma}

\begin{lemma}\label{lem3.15}
Let $L$ be a generalized Heisenberg Lie superalgebra of rank at least two, and let $L = H \diamond K$. Then the following results hold:
\begin{enumerate}
\item $H$ and $K$ are generalized Heisenberg Lie superalgebras.
\item 	$f_{L}$ is a decomposable nondegenerate object of $\mathcal{SSKE}$. 
\end{enumerate}
\end{lemma}

\begin{lemma}\label{lem3.16}
Let $f:V\times V\longrightarrow W$ be a decomposable nondegenerate object of rank two in $\mathcal{SSKE}$ with the notations used in Definition \ref{def3222}. If $\overline W_i=W$ for each $i = 1, 2$, then $Z^*(f) =W$.
\end{lemma}

\begin{lemma}\label{lem3.17}
Let $f:V\times V\longrightarrow W$ be a decomposable nondegenerate object of rank two in $\mathcal{SSKE}$ with the notations used in Definition \ref{def3222}. Assume that $\overline W_1=W$ and $\dim \overline W_2=1$. Then $Z^*(f) =W$ or $Z^*(f) =\overline W_2$.
\end{lemma}

\begin{theorem}\label{the3.20}
Let $f:V\times V\longrightarrow W$ be a decomposable nondegenerate object of rank two in $\mathcal{SSKE}$ with the notations used in Definition \ref{def3222}. Assume that $\dim \overline W_i=(r_i|s_i)$ where $r_i+s_i=1$ for each $i =1, 2$. Then $Z^*(f) =Z^*(\overline f_1) \oplus Z^*(\overline f_2)$ and

$$
	L_{f}\cong
	\begin{cases}
		H(m|n)	\oplus	H(m'|n') \quad \mbox{if}\;r_i=1,~s_i=0\\
		H(l) \oplus	H(l') \quad \mbox{if}\;r_i=0,~ s_i=1\\  
		H(m|n)\oplus H(l) \quad \mbox{if}\; r_1=1, ~s_1=0;~r_2=0, ~s_2=1.\\  
	\end{cases}
$$
for some $m,~n,~l,~m',~n',~l'$.
\end{theorem}
\begin{proof}
First we show that $\overline f_i$ is nondegenerate. Suppose that there is an element $x \in \overline V_i$ such that $\overline f_i(x, y) =0$ for all $y\in \overline V_i$. As $f$ is decomposable, we have $f(x, v) =0$ for all $v \in V$, then we get a contradiction. Therefore, $\overline W_i \neq 0$ for each $i =1, 2$. Theorem \ref{the3.14}(i,ii) implies that 
		$$
	L_{f_1}\cong
	\begin{cases}
		H(m|n) \quad \mbox{if}\;r_1=1,~s_1=0\\
			H(l) \quad \mbox{if}\;r_1=0, ~s_1=1\\  
	\end{cases}
	$$
	
	for some $m,~n,~l$, and 
		$$
	L_{f_2}\cong
	\begin{cases}
		H(m'|n') \quad \mbox{if}\;r_2=1,~s_2=0\\
		H(l') \quad \mbox{if}\;r_2=0, ~s_2=1.\\  
	\end{cases}
	$$
for some $m',~n',~l'$. Since $W_1\cap W_2=0$, Lemma \ref{lem3.161} shows that $Z^*(f) =Z^*(f_1) \oplus Z^*(f_2)$ and $L_{f}= L_{f_1}
	\oplus L_{f_2}$, therefore 
		$$
	L_{f}\cong
	\begin{cases}
	H(m|n)	\oplus	H(m'|n') \quad \mbox{if}\;r_i=1,~s_i=0\\
		H(l) \oplus	H(l') \quad \mbox{if}\;r_i=0, ~s_i=1\\  
		H(m|n)\oplus H(l) \quad \mbox{if}\; r_1=1,~ s_1=0;~r_2=0, ~s_2=1.\\  
	\end{cases}
	$$

\end{proof}

Now we describe the epicenter of a central sum of two generalized Heisenberg Lie algebras of rank two.

\begin{theorem}
Let $L$ be a generalized Heisenberg Lie superalgebra of rank two and let $L = H \diamond K$. Then the following results hold:
	
	\begin{enumerate}
		\item If $\dim H^2=\dim K^2=2$, then $Z^*(L) = Z(L)$.
		\item If $\ dim H^2=\ dim K^2=1$, then 
		$$
		L_{f}\cong
		\begin{cases}
			H(m|n)	\oplus	H(m'|n') \quad \mbox{if}\;\ dim H^2=\ dim K^2=(1|0)\\
			H(l) \oplus	H(l') \quad \mbox{if}\;\ dim H^2=\ dim K^2=(0|1)\\  
			H(m|n)\oplus H(l) \quad \mbox{if}\; \ dim H^2=(1|0), \ dim K^2=(0|1),\\  
		\end{cases}
		$$
		and  $Z^*(L) =Z^*(H) \oplus Z^*(K)$ for some $m,~n,~l,~n',~m',~l'$.
       \item 	If $\ dim H^2=2$ and $\ dim K^2=1$, then $Z^*(L) =Z(L)$ or $Z^*(L)=K^2$.
\end{enumerate}
\end{theorem}
\begin{proof}
The result follows from Theorem \ref{the3.20}, Lemmas \ref{lem3.15},  \ref{lem3.16}, and \ref{lem3.17}.
\end{proof}


\begin{thebibliography}{1}\label{reference}

\bibitem{Baer1938}
	Baer, B. (1938).
	\newblock  Groups with preassigned central and central quotient group.
	\newblock {\em Trans. Amer. Math. Soc}. 44(3): 387-412. DOI:  10.1090/S0002-9947-1938-1501973-3





%\bibitem{Beyl1979}
%	Beyl, F., Felgner, U., Schmid, P. (1979).
%	\newblock  On groups occurring as center factor groups.
%	\newblock {\em J. Algebra}. 61(1): 161-177. DOI:  10.1016/0021-8693(79)90311-9
	
	
	
	
	
\bibitem{Ellis1996}
	Ellis,G. (1996).
	\newblock  Capability, homology and central series of a pair of groups.
	\newblock {\em J. Algebra}. 179(1): 31-46. DOI:  10.1006/jabr.1996.0002






\bibitem{GHLS2018}
	Padhan, Rudra Narayan, and K. C. Pati. (2018)
	\newblock Capability and the Schur multiplier of generalized Heisenberg Lie superalgebras.
	arXiv preprint arXiv:1810.04862
	
	
	
	
	
	
\bibitem{2Nilpotent2020}
	Padhan, Rudra Narayan, Nupur Nandi, and K. C. Pati. (2020)
	\newblock On $2 $-Nilpotent Multiplier of Lie Superalgebras.
	 arXiv preprint arXiv:2006.10970
	
	
	
	
	
	
	\bibitem{Alamian2008}
	Alamian, V., Mohammadzadeh, H., Salemkar, A. (2008)
	\newblock Some properties of the Schur multiplier and covers of Lie algebras.
	\newblock {\em Comm. Algebra}. 36(2): 697-707. DOI:
	10.1080/00927870701724193
	
	
	
	
	
	
	\bibitem{Padhandetec}
	Padhan, R. N., Nayak, S., Pati, K. C. (2021).
	\newblock  Detecting Capable Lie superalgebras.
	\newblock {\em Comm. Algebra}. 49(10): 4274-4290. DOI:
	10.1080/00927872.2021.1918135
	
	

	
	\bibitem{Padhantrip}
	Hasan, I. Y., Padhan, R. N. (2022).
	\newblock  On the triple tensor product of generalized Heisenberg Lie superalgebra of rank $\leq2$.
	\newblock arXiv preprint arXiv:2210.00254
	
	
	
	
	
	\bibitem{Padhantensor}
	Padhan, R. N., Hasan, I. Y., Pradhan, S. S. (2022).
	\newblock  On the dimension of non-abelian tensor square of Lie superalgebras.
	\newblock arXiv preprint arXiv:2209.09514
	
	
	
	
	\bibitem{PadhanISO}
	Padhan, R. N., Khuntia, T. K. (2022).
	\newblock  Isoclinic relative central extensions of a pair of regular Hom-Lie algebras.
	\newblock arXiv preprint arXiv:2209.07822
	
	
	
	
	\bibitem{PadhanHASDET}
	Hasan, I. Y., Padhan, R. N., Das, M. (2022).
	\newblock  Detecting capable pairs of some nilpotent Lie superalgebras.
	\newblock {\em  Indian J Pure Appl Math }. DOI:10.1007/s13226-022-00348-0
	
	
	
	
	
	\bibitem{PadhanCOHOM}
	Nandi, N., Padhan, R. N. (2022).
	\newblock  Cohomology, superderivations, and abelian extensions of $3 $-Lie superalgebras.
	\newblock arXiv preprint arXiv:2207.11443
	

	

	
	
\bibitem{PadhanHOM}
	Padhan, R. N., Nandi, N., Pati, K. C. (2020).
	\newblock  Some properties of factor set in regular Hom-Lie algebras.
	\newblock arXiv preprint arXiv:2005.05555
	
	
	
	
	
	\bibitem{Padhan2MULT}
	Padhan, R. N., Pati, K. C. (2018).
	\newblock  On dimension of the 2-multiplier of nilpotent Lie algebras.
	\newblock arXiv preprint arXiv:1811.09771
	
	
	
	
	
	\bibitem{PadhanSPLIT}
	Padhan, R. N., Pati, K. C. (2017).
	\newblock  Splints of root systems of Lie superalgebras.
	\newblock arXiv preprint arXiv:1707.06454
	
	
	
	%\bibitem{YAM2000}
%	Bahturin, Y., Mikhalev, a., Petrogradsky, v. m. et al. (1992).
%	\newblock {\em Infinite dimensional Lie superalgebras}
%	\newblock De Gruyter Expositions in Mathematics \textbf{7}) , Berlin, Germany: Walter de Gruyter.
	
	
	
	
	
\bibitem{Johari21}
	Johari, F. (2021).
	\newblock  Finite-dimensional nilpotent Lie algebras of class two and alternating bilinear maps.
	\newblock {\em Journal of Algebra}. 586(10): 561-581. DOI:10.1016/j.jalgebra.2021.06.040
	
	

	
	
	\bibitem{Batten1993} 
	Batten, P. (1993).
	\newblock  Multipliers and covers of Lie algebras. 
	\newblock Ph.D dissertation. North Carolina State University, Raleigh, North Carolina.
	
	
	
	
	
	\bibitem{Batten1996}
	Batten, P., Moneyhun, K., Stitzinger, E. (1996).
	\newblock On characterizing nilpotent Lie algebra by their multipliers.
	\newblock {\em Comm. Algebra}. 24(14): 4319-4330. DOI: 10.1080/00927879608825817
	
	
	
	
	
	
	\bibitem{BS1996} 
	Batten, P., Stitzinger, E. (1996).
	\newblock On covers of Lie algebras.
	\newblock  {\em Comm. Algebra}. 24(14): 4301-4317. DOI:
	10.1080/00927879608825816
	
	
	
	
	
	
	\bibitem{Beyl1979}
	Beyl, F., Felgner, U., Schmid, P. (1979).
	\newblock  On groups occurring as center factor groups.
	\newblock {\em J. Algebra}. 61(1): 161-177. DOI:  10.1016/0021-8693(79)90311-9
	
	
	
	
	
	
	\bibitem{Tappe}
	Beyl, F., Tappe, J. (1982).
	\newblock  {\em Groups Extensions and Representations and the Schur Multiplier}. 
	\newblock Lecture Notes in Mathematics, Vol. 958. New York, Berlin: Springer. 
	
	
	
	
	
	
	\bibitem{p1}
	Bosko, L. (2010).
	\newblock On Schur multiplier of Lie algebras and groups of maximal class.
	\newblock {\em Internat. J. Algebra Comput}. 20(6): 807-821. DOI: doi.org/10.1142/S0218196710005881
	
	
	
	
	
	
	%\bibitem{Brown2010}
%	Brown, R., Loday, J.L. (2010).
	%\newblock  Van Kampen theorems for diagrams of spaces.
%	\newblock {\em Topology}. 26(3): 311-335. DOI: 10.1016/0040-9383(87)90004-8
	
	
	
	
	
	
	
	\bibitem{Ellis1987}
	Ellis, G. (1987).
	\newblock Non-abelian exterior products of Lie algebras and an exact sequence in the homology of Lie algebras.
	\newblock {\em J. Pure Appl. Algebra}. 46(2-3): 111-115.
	
	
	
	
	
	
	
	
	\bibitem{Ellis1991}
	Ellis, G. (1991).
	\newblock  A non-abelian tensor products of Lie algebras.
	\newblock {\em Glasg. Math. J}. 33(1): 101-120. DOI: 10.1017/S0017089500008107
	
	
	
	
	
	\bibitem{Ellis1995}
	Ellis, G. (1995).
	\newblock Tensor products and $q$-cross modules.
	\newblock {\em J. London Math. Soc}. 51(2): 241-258. DOI: 10.1112/JLMS/51.2.243
	
	

	
	
	
	
	\bibitem{GKL2015} 
	Garc\'{i}a-Mart\'{i}nez, X., Khmaladze, E., Ladra, M. (2015).
	\newblock Non-abelian tensor product and homology of Lie superalgebras.
	\newblock {\em J. Algebra}. 440: 464-488.
	DOI:
	10.1016/j.jalgebra.2015.05.027
	
	
	
	
	
	
	
	%\bibitem{Hall}
	%Hall, p. (1940).
   %  \newblock The classification of prime power groups.
	%\newblock {\em J. reine Angew. Math}. 182: 130-141. DOI: 10.1515/crll.1940.182.130
	
	
	
	
	
	
	\bibitem{Hardy1998}
	Hardy, P., Stitzinger, E. (1998).
	\newblock On characterizing nilpotent Lie algebras by their multipliers $t(L)=3, 4, 5, 6$.
	\newblock {\em Comm. Algebra }. 26(11): 3527-3539. DOI:
	10.1080/00927879808826357
	
	
	
	
	
	
	
	\bibitem{Hardy2005}
	Hardy, P. (2005).
	\newblock On characterizing nilpotent Lie algebras by their multipliers III.
	\newblock {\em Comm.  Algebra }. 33(11): 4205-4210. DOI:
	10.1080/00927870500261512
	
	
	
	
	
	
	
	\bibitem{org}
	Johari, F., Parvizi, M., Niroomand, P. (2017).
	\newblock  Capability and Schur multiplier of a pair of Lie algebras.
	\newblock {\em Journal of Geometry and Physics}. 114: 184-196. DOI: 10.1016/J.GEOMPHYS.2016.11.016
	
	
	
	
	
	
	
	\bibitem{Kac1977}
	Kac, V.G. (1977).
	\newblock  Lie superalgebras.
	\newblock {\em Adv. Math}. 26(1): 8-96. DOI:
	10.1016/0001-8708(77)90017-2
	
	
	
	
	
	
	\bibitem{Kar1987}
	Karpilovsky, G. (1987).
	\newblock {\em The Schur Multiplier }. 
	\newblock  London, DC: London Math. Soc.
	
	
	
	
	\bibitem{p52}
	Khuntia, T. K., Padhan, R. N., Pati, K. C. (2022).
	\newblock Inner Superderivations of $n$-Isoclinism Lie superalgebras.
	\newblock {\em Results Math}. 77(3) .
	DOI: 10.1007/s00025-022-01643-2
	
	
	
  %  \bibitem{Moneyhun1994}
	%Moneyhun, K. (1994).
	%\newblock Isoclinisms in Lie algebras.
	%\newblock {\em Algebras Groups Geom}. 11: 9-22.
	
	
	
	\bibitem{Musson2012}
	Musson, I. (2012).
	\newblock {\em Lie  superalgebras  and  enveloping  algebras}.
	\newblock 
	Graduate Studies in
	Mathematics, Vol. 131. Milwaukee, Wisconsin: Amer. Math. Soc. 
	
	
	
	
	
    \bibitem{p51}
	Nandi, N., Padhan, R. N., Pati, K. C. (2022).
	\newblock Superderivations of direct and semidirect sum of Lie superalgebras.
	\newblock {\em Comm. Algebra}. 50(3): 1055-1070. DOI: 10.1080/00927872.2021.1977943
	
	
	
	
	
	\bibitem{Nayak2018} 
	Nayak, S. (2019).
	\newblock Multipliers of nilpotent Lie superalgebras.
	\newblock {\em Comm. Algebra}. 47(2): 689-705. DOI:  10.1080/00927872.2018.1492595
	
	
	\bibitem{SN2018b}
	Nayak, S. (2021).
	\newblock Classification of finite dimensional nilpotent Lie superalgebras by their multiplier.
	\newblock {\em J. Lie Theory}. 31(2): 439-458.
	
	
	
	\bibitem{NPP} 
	Nayak, S., Padhan, R. N., Pati, K. C. (2020).
	\newblock Some properties of isoclinism in Lie superalgebras.
	\newblock {\em Comm. Algebra} 48(2): 523-537. DOI:
	10.1080/00927872.2019.1648654
	
	
	\bibitem{N2}
	Niroomand, P. (2011).
	\newblock On the dimension of the Schur multiplier of nilpotent Lie algebras.
	\newblock {\em Cent. Eur. J. Math}. 9(1): 57-64. DOI:
	10.2478/s11533-010-0079-3
	
	
	
	\bibitem{Niroomand2011}
	Niroomand, P., Russo, F. G. (2011).
	\newblock  A note on the Schur multiplier of a nilpotent Lie algebra.
	\newblock {\em Comm. Algebra}. 39(4): 1293-1297. DOI:
	10.1080/00927871003652660
	
	
	
	\bibitem{Russo2011}
	Niroomand, P., Russo, F. G. (2011).
	\newblock  A restriction on the Schur multiplier of a nilpotent Lie algebras.
	\newblock {\em Electron. J. Linear Algebra}. 22(1): 1-9. DOI:
	10.13001/1081-3810.1423
	
	\bibitem{PMF2013}
	Niroomand, P., Parvizi, M., Russo, F. G. (2013).
	\newblock  Some criteria for detecting capable Lie algebras.
	\newblock {\em J. Algebra}. 384: 36-44. DOI:
	10.1016/j.jalgebra.2013.02.033
	
	
	
	\bibitem{p50}
	Padhan, R. N., Nayak, S. (2022).
	\newblock  On capability and the Schur multipliers of some nilpotent Lie superalgebras.
	\newblock {\em Linear Multilinear }. 70(8): 1467-1478. DOI:
	10.1080/03081087.2020.1764902
	
	
	%\bibitem{Padhandetec}
	%Padhan, R. N., Nayak, S., Pati, K. C. (2021).
	%\newblock  Detecting Capable Lie superalgebras.
	%\newblock {\em Comm. Algebra}. 49(10): 4274-4290. DOI:
	%10.1080/00927872.2021.1918135
	
	
	\bibitem{p53}
	Padhan, R. N., Pati, K. C. (2020).
	\newblock Some studies on central derivation of nilpotent Lie superalgebra.
	\newblock {\em Asian-Eur. J. Math}. 13(4): . 2050068.
	DOI:
	10.1142/S1793557120500680
	
	
	
	%\bibitem{pet2000}
	%Petrogradsky, V. M. (2000)
	%\newblock On Witts's formula and invariants for free Lie superalgebras,
	%\newblock {\em Formal Power Series and Algebraic Combinatorics}. 543-551.
	%DOI:
	%10.1007/978-3-662-04166-6-52
	
	
	\bibitem{p10}
	Pourmirzaei, A., Hokmabadi, A., Kayvanfar, S. (2011).
	\newblock Capability of a pair of groups.
	\newblock {\em Bull. Malays. Math. Sci. Soc}. 35(1): 205-213.
	
	\bibitem{MC2011}
	Rodriguez-Vallarte, M. C., Salgado, G., Sanchez-Valenzuela, O. A. (2011).
	\newblock Heisenberg Lie superalgebras and their invariant superorthogonal and supersympletic forms.
	\newblock {\em J. Algebra}. 332(1): 71-86.
	DOI:
	10.1016/j.jalgebra.2011.02.003
	
	
	\bibitem{hesam}
	Safa, H. (2021)
	\newblock On multipliers of pairs of Lie superalgebras.
	\newblock {\em Quaestiones Mathematicae}. 1-12. DOI:
	10.2989/16073606.2021.1978010
	
	
	
	
	
	
	
\end{thebibliography}
\end{document}